\documentclass[11pt]{amsart}
\usepackage{amsmath}
\usepackage{amsthm}
\usepackage{amssymb}
\usepackage{mathtools}
\usepackage{hyperref}
\usepackage{upref}
\usepackage[capitalize]{cleveref}
\usepackage{color}
\usepackage{comment}
\usepackage{siunitx}

\numberwithin{equation}{section}
\newtheorem{lem}[equation]{Lemma}
\newtheorem{thm}[equation]{Theorem}
\newtheorem{claim}{Claim}
\newtheorem{conj}[equation]{Conjecture}

\theoremstyle{definition}
\newtheorem{defn}[equation]{Definition}

\theoremstyle{remark}
\newtheorem{obs}[equation]{Observation}
\newtheorem*{ack}{Acknowledgments}

\crefformat{lem}{Lemma~#2#1#3}
\Crefname{lem}{Lemma}{Lemmas}
\crefname{lem}{lemma}{lemmas}
\crefformat{thm}{Theorem~#2#1#3}
\Crefname{thm}{Theorem}{Theorems}
\crefname{thm}{theorem}{theorems}

\newcommand{\pref}[1]{\textup(\ref{#1}\textup)}

\newcommand{\fullcref}[2]{\cref{#1}\pref{#1-#2}}

\newcounter{case}
\newenvironment{case}[1][\unskip]{\refstepcounter{case}\bf
\medskip \noindent Case \thecase\ #1.\ \it}{\unskip\upshape}
\numberwithin{case}{equation}
\renewcommand{\thecase}{\arabic{case}}
\crefformat{case}{Case~#2\textup{#1}#3}
\Crefname{case}{Case}{Cases}
\crefname{case}{case}{cases}

\newcommand{\NN}{\mathbb{N}}

\DeclareMathOperator{\Aut}{Aut}

\DeclareMathOperator{\SymGrp}{S}
\DeclareMathOperator{\AltGrp}{A}
\DeclareMathOperator{\Perm}{Perm}

\DeclareMathOperator{\Exp}{Exp}

\title[Exponential graph growth]{On transitive permutation groups with exponential graph growth}

\author{\DJ or\dj e Mitrovi\'c}
\address{Department of Mathematics, University of Auckland}
\email{dmit755@aucklanduni.ac.nz}

\author{Gabriel Verret}
\address{Department of Mathematics, University of Auckland}
\email{g.verret@auckland.ac.nz}


\begin{document}

\begin{abstract}
    Let $\Gamma$ be a finite connected graph and $G$ a vertex-transitive group of its automorphisms. The pair $(\Gamma, G)$ is said to be locally-$L$ if the permutation group induced by the action of the vertex-stabiliser $G_v$ on the set of neighbours of a vertex $v$ in $\Gamma$ is permutation isomorphic to $L$. The maximum growth of $|G_v|$ as a function of $|V\Gamma|$ for locally-$L$ pairs $(\Gamma,G)$ is called the graph growth of $L$. We prove that if $L$ is a transitive permutation group on a set $\Omega$ admitting a nontrivial block $B$ such that the pointwise stabiliser of $\Omega\setminus B$ in $L$ is nontrivial, then the graph growth of $L$ is exponential. This generalises several results in the literature on transitive permutation groups with exponential graph growth.
\end{abstract}

\maketitle

\section{Introduction}

All groups considered in this paper are finite, and all graphs are finite and simple (undirected, with no loops or multiple edges). An \emph{arc} of a graph $\Gamma$ is an ordered pair of adjacent vertices. We use $V\Gamma$ and ${A\Gamma}$  to denote the set of vertices and arcs of $\Gamma$, respectively.  The \emph{automorphism group} $\Aut(\Gamma)$ of $\Gamma$ is the group of all permutations of $V\Gamma$ preserving $A\Gamma$. For $G\leq \Aut(\Gamma)$, $\Gamma$ is said to be \emph{$G$-vertex-transitive} and \emph{$G$-arc-transitive} if $G$ is transitive on $V\Gamma$ and $A\Gamma$, respectively. We let $G_v$ denote the stabiliser of a vertex $v$ in $G$.

Tutte showed in \cite{Tutte1947_FamilyOfCubicGraphs,Tutte1959_SymmetryOfCubicGraphs} that if $\Gamma$ is a connected $G$-arc-transitive $3$-valent graph and $v\in V\Gamma$, then $|G_v|\leq 48$. This is one of the most important results in algebraic graph theory, with many different applications. For example, in \cite{ConderDobcsanyi786}, the authors used this to compile a complete list of $3$-valent arc-transitive graphs on at most $768$ vertices, a census that has subsequently been extended to include graphs with up to $10000$ vertices \cite{Conder10000}. Tutte's theorem was also used in asymptotic enumeration of  $3$-valent arc-transitive graphs {\cite{AsymptoticEnumerationFixedValency}}.

Given the importance of Tutte's theorem, it is natural to seek to generalise it. To discuss these generalisations, we must first introduce the notion of local action. Let $\Gamma$ be a graph, let $G\leq\Aut(\Gamma)$ and let $v\in V\Gamma$. Note that $G_v$ acts  on the neighbourhood $\Gamma(v)$ of $v$  inducing a permutation group which we denote by $G_v^{\Gamma(v)}$ and call the \emph{local action} of $G$ at $v$. If $\Gamma$ is connected and $G$-vertex-transitive and $G_v^{\Gamma(v)}$ is permutation isomorphic to $L$, then $(\Gamma,G)$ is said to be a \emph{locally-$L$} pair. Note that this is independent of the choice of $v$. Note also that $\Gamma$ is $G$-arc-transitive if and only if $L$ is transitive.

A permutation group $L$ for which there exists a constant $c(L)$ such that $|G_v|\leq c(L)$ for all locally-$L$ pairs $(\Gamma,G)$ is said to be \textit{graph-restrictive}. With this terminology, Tutte's theorem can be rephrased as saying that the transitive permutation groups of degree $3$ are graph-restrictive, and in fact one can take $c(\AltGrp_3)=3$ and $c(\SymGrp_3)=48$. Quite a lot of work has been done over the years on graph-restrictive groups \cite{Gardiner1973,LukeMorgan-Frobenius,MorganGuidici-LocallySemiPrimGraphs,GiudiciMorgan-TheoryOfSemiPrimitiveGroups, Morgan_SnOnOrderedPairs, PSV-GraphRestrictive, Spiga-TwistedWreathType,Spiga-AffineType,IntransitiveGraphRes,OnTheWeissConjectureI_Trofimov,TrofimovWeiss-LocallyLinearGraphs,TrofimovWeiss-E6q, VerretPSubRegular,VerretWeaklyPSubRegular}. For an overview of some of these results see the following surveys \cite{PSV-GraphRestrictive,SpigaOverview2025}. These results, much like Tutte's theorem, lead to many applications \cite{Darius-DensityOfQuotientOrders,Fixity_FlorianPrimozPablo,2ATofOrderkpn,4val2ATPrimozPotocnik}.

Recently, there has been a growing interest in local actions that are not graph-restrictive. The maximum growth of the order of  vertex-stabilisers as a function of $|V\Gamma|$ for locally-$L$ pairs $(\Gamma,G)$ is called the \textit{graph growth} of $L$. (This notion was introduced in \cite{PSV-OrderOfArcStabilizer} under the name graph-type.) Using this language, graph-restrictive permutation groups are precisely those with constant graph growth. Many of the previously described applications to groups with constant graph growth can be extended to groups with ``slow'' graph growth, for example groups with polynomial graph growth. Our focus in this paper is  permutation groups to which these methods do not apply.

\begin{defn}[\cite{PSV-OrderOfArcStabilizer}]\label{ExpGraphGrowthDefn}
    A permutation group $L$ is said to have \textit{exponential graph growth} if there exists a family of locally-$L$ pairs $\{(\Gamma_n,G_n)\}_{n\in\NN}$ such that $\lim_{n\rightarrow \infty}|V\Gamma_n|=\infty$, and a constant $c>1$ satisfying \[|(G_n)_v|\geq c^{|V\Gamma_n|},\forall v \in V\Gamma_n, \forall n\in \NN.\]
\end{defn}

The main contribution of this paper is to describe a large new  family of transitive permutation groups with exponential graph growth. Given a permutation group $L$ on a set $\Omega$ and a subset $S\subseteq \Omega$,  we denote the point-wise stabiliser of $S$ in $L$ by $L_{(S)}$.

\begin{thm}\label{MainThm}
    Let $L$ be a transitive permutation group on a set $\Omega$ admitting a nontrivial block $B$. If $L_{(\Omega \setminus B)}\neq 1$, then $L$ has exponential graph growth.
\end{thm}

\cref{MainThm} generalises  {\cite[Theorem 7 and Corollary 9]{PSV-OrderOfArcStabilizer}}, which were the only previously known sources of infinitely many examples of transitive permutation groups with exponential graph growth. The only other previously known examples of transitive  permutation groups with exponential growth were three groups of degree $6$ given by {\cite[Theorem 1.1]{ThreeLocalActionsDegree6}}. Based on this and some other preliminary work, we present the following conjectured characterisation of transitive permutation groups with exponential growth.

    \begin{conj}\label{ExpConj}
    A transitive permutation group $L$ has  exponential graph growth if and only if $L$ admits a proper block $B$ such that \[N_{(B)}\neq 1\] where $N$ is the kernel of the action of $L$ on the block system induced by $B$.
    \end{conj}

In {\cite[Table 1]{PSV-OrderOfArcStabilizer}}, the authors provided an overview of the state of knowledge of graph growth of transitive permutation groups of degree at most $7$. In \cref{Table}, we provide an updated version of this table for transitive permutation groups of degree at most $47$. 

     \begin{table}
        \caption{\small Graph growth of transitive permutation groups of small degree}
        \label{Table}
        \scalebox{0.8}{
        \begin{tabular}{c|c|c|c|c|c}
         Degree & Non-Exp & Prev Exp &	Current Exp &	Unknown	& Total\\
          \hline
          2	& 1	& 0	& 0	& 0	& 1 \\
          3 & 2 & 0	& 0	& 0	& 2  \\
          4	& 4	& 1	& 1 & 0	& 5  \\
          5	& 5	& 0	& 0	& 0	& 5  \\
          6	& 7	& 9	& 9	& 0 & 16 \\
          7	& 7 & 0	& 0	& 0	& 7 \\
          8	& 18 & 28 & 28 & 4 & 50 \\
          9	& 15 & 4 & 10 &	9 & 34 \\
          10 & 12 & 20 & 20 & 13 & 45 \\
          11 & 8 & 0 & 0 & 0 & 8  \\
          12 & 17 & 138 & 218 & 66 & 301\\
          13 & 9 & 0 & 0 & 0 & 9 \\
          14 & 9 & 28 & 28 & 26 & 63 \\
          15 & 6 & 20 & 60 & 38 & 104 \\
          16 & 48 & 1638 & 1733	& 173 & 1954 \\
          17 & 10 & 0 & 0 & 0 & 10  \\
          18 & 13 &	453	& 830 &	140 & 983 \\
          19 & 8 & 0 & 0 & 0 & 8 \\
          20 & 14 & 615 & 862 & 241 & 1117 \\
          21 & 10 & 28 & 104 & 50 & 164 \\
          22 & 11 & 30 & 30 & 18 & 59 \\
          23 & 7 & 0 & 0 & 0 & 7 \\
          24 & 38 & 12593 & 23749 & 1213 & 25000 \\
          25 &	35 & 25	& 95 & 81 &	211 \\
          26 & 13 & 58 & 58 & 25 & 96 \\
          27 & 26 &	128	& 1822 & 544 & 2392 \\
          28 & 22 &	812	& 1407 & 425 & 1854 \\
          29 &	8 & 0 & 0 & 0 & 8 \\
          30 & 11 & 2010 &  4916 & 785& 5712 \\
          31 & 12 &	0 & 0 & 0 & 12 \\
          32 & 77 & 2783459 & 2797566 & 3681 & 2801324 \\
          33 & 6 & 32 & 94 & 62 & 162 \\
          34 & 6 & 68 & 68 & 41 & 115 \\
          35 & 6 & 70 & 304 & 97 & 407 \\
          36 & 36 &	55835 & 118010 & 3233 & 121279\\
          37 & 11 & 0 & 0 & 0 & 11 \\
          38 & 9 & 48 & 48 & 19 & 76 \\
          39 & 5 & 36 & 174 & 127 & 306 \\
          40 & 29 & 144501 & 308656 & 7157 & 315842 \\
          41 & 10 & 0 & 0 & 0 & 10 \\
          42 & 23 & 3677 &  8599 &  869 & 9491 \\
          43 & 10 & 0 & 0 & 0 & 10 \\
          44 & 10 & 784 & 1703 & 400 & 2113\\
          45 & 8 & 696 & 9946 & 969 & 10923 \\
          46 & 7 & 28 & 28 & 21 & 56 \\
          47 & 6 & 0 & 0 & 0 & 6 \\
          \hline 
          TOTAL (\#)	& 665 & 3007872 &  3281176 &  20527 & 3302368\\
          TOTAL (\%) &	0.020 &	91.08 & 99.36 & 0.62 & 100
          \end{tabular}}
\end{table}

The first column of \cref{Table} (Degree) lists the degrees of transitive permutation groups we are considering, while the last column (Total) gives the number of transitive permutation groups of the corresponding degree (up to permutation isomorphism). These numbers were obtained from the  database of transitive permutation groups of small degree  \cite{HoltCannon_TransPermGroupsDataBase32,HoltRoyle_TransPermGroupsDataBase48,Hulpke_TransPermGroupsDataBase<=31} available in MAGMA \cite{MAGMA}.

The second column (Non-Exp) counts the number of transitive permutation groups of the corresponding degree whose graph growth is known to be non-exponential. This was obtained by running through the database of transitive groups and applying results on graph-restrictive groups from articles referenced earlier in this section,  as well as {\cite[Theorem A]{VerretWeaklyPSubRegular}} and {\cite[Corollary 9]{PSV-OrderOfArcStabilizer}} (for groups with polynomial graph growth). The next column  (Prev Exp) counts the number of groups whose graph growth was known to be exponential prior to our current work, from {\cite[Theorem 7]{PSV-OrderOfArcStabilizer}}, {\cite[Corollary 9]{PSV-OrderOfArcStabilizer}} or {\cite[Theorem 1.1]{ThreeLocalActionsDegree6}}. The fourth column (Current Exp)  counts the number of groups whose graph growth is now known to be exponential, via \cref{MainThm} or {\cite[Theorem 1.1]{ThreeLocalActionsDegree6}} while the fifth column (Unknown) counts the groups with unknown graph growth.

    \begin{obs}\label{99}
        The following observations can be derived from  \cref{Table}.
        \begin{enumerate}
            \item \label{99-all} Among the \num{3302368} transitive permutation groups of degree at most $47$, \num{3281199} (that is, $99.36\%$ of the total)  are now  known to have exponential graph growth. Moreover, with the exception of three transitive permutation groups of degree $6$ described in \cite{ThreeLocalActionsDegree6}, all of these groups satisfy the hypothesis of \cref{MainThm}.
            \item \label{99-analysis} Previously known results already seemed to suggest that most transitive permutation groups with degree a power of $2$ have exponential graph growth (at least $99.36\%$ of all groups of degree $32$ for example). However, for degrees that are not a power of $2$ yet have many divisors (such as $24$, $36$, $40$, etc.) and thus admit many transitive permutation groups, the situation was less clear. Indeed, in these cases, previously known results only applied to about half the groups, or even fewer. Using \cref{MainThm}, we can now conclude that the vast majority of the groups of these degrees have exponential graph growth. (At least $97.7\%$ of all groups of degree $40$ for example.) 
        \end{enumerate}
    \end{obs}

Motivated by \cref{99}, we conjecture the following. 
    
    \begin{conj}\label{AlmostAllAreExpConj}
        Almost all transitive permutation groups have exponential graph growth, that is
        \[\lim_{n\rightarrow \infty}\frac{\Exp(n)}{\Perm(n)} = 1, \]
        where $\Perm(n)$ is the number of isomorphism classes of transitive permutation groups of degree at most $n$, and  $\Exp(n)$ is the number of isomorphism classes of transitive permutation groups of degree at most $n$ with exponential graph growth. 
    \end{conj}

The rest of the paper is dedicated to the proof of \cref{MainThm}, given in \cref{ProofOfMainThm}, with some preliminaries in \cref{SectionCoversOfKnn}.

\section{Some families of covers of complete bipartite graphs}\label{SectionCoversOfKnn}

The proof of \cref{MainThm} in \cref{ProofOfMainThm} relies on the existence of certain graph-group pairs which we now construct.

\begin{lem}\label{Gamma} Let $n\geq 2$ and let $L$ be a transitive permutation group on $\{1,\ldots,n\}$. There exist graph-group pairs $(\Gamma,G)$ with $|V\Gamma|$ arbitrarily large such that
    \begin{enumerate}
        \item \label{Gamma-Bipartite} $\Gamma$ is connected, $n$-valent and bipartite with bipartition sets $X$ and $Y$;
        \item \label{Gamma-LocalAction} $(\Gamma,G)$ is locally-$L$;
        \item \label{Gamma-P} $V\Gamma$ admits a $G$-invariant partition $\mathcal{P}$ into $2n$ sets $X_1,\ldots,X_n$ and $Y_1,\ldots,Y_n$ such that \[\text{$\bigcup_{i=1}^{n} X_i = X$ and $\bigcup_{i=1}^{n} Y_i = Y$,}\]
        and every vertex $v\in X$ has precisely one neighbour in each of the sets $Y_i$ with $1\leq i \leq n$;
        \item \label{Gamma-flip} There exists $f\in G$ such that $f^2 = 1$, $X_i^f = Y_i$ and $Y_i^f  = X_i$ for all $1\leq i\leq n$;
        \item \label{Gamma-S} If $S$ denotes the kernel of the action of $G$ on $\{X,Y\}$, then
            \begin{enumerate}
                \item \label{Gamma-S-f} $G = \langle S,f\rangle$ and $S^f = S$;
                \item \label{Gamma-S-S^p} $S^\mathcal{P}$ is permutation isomorphic to $L\times L$ $($in its natural intransitive action on $\{1,\ldots,n\}\times \{1,\ldots,n\})$;
                \item \label{Gamma-S-Stabiliser} For every $v\in X$ and every $\pi\in L$, there exists $s\in S$ such that $v^s = v$, $X_i^s = X_i$ and $Y_i^s = Y_{i^\mathcal{\pi}}$ for all $1\leq i\leq n$.
            \end{enumerate}
    \end{enumerate}
\end{lem}

\begin{proof}
Let $X$ and $Y$ be disjoint sets of size $n$ and let $\Gamma$ be the complete bipartite graph with bipartition sets $X$ and $Y$. Labelling $X$ and $Y$ with $\{1,\ldots,n\}$, we get a natural action of $L\times L$ on $X\cup Y$. Let $f$ be the involution of $X\cup Y$ that maps an element of $X$ to the element of $Y$ with the same label and vice-versa. Let  $G = (L\times L)\rtimes \langle f \rangle\leq\Aut(\Gamma)$.

We first show that this particular pair $(\Gamma,G)$ satisfies \pref{Gamma-Bipartite}--\pref{Gamma-S} and then we will show how to construct infinitely many such pairs using covers of this starting pair.

Part \pref{Gamma-Bipartite} clearly holds. Let $X_1,\ldots,X_n$ be the singleton subsets of $X$ and   $Y_1,\ldots,Y_n$ be the singleton subsets of $Y$. The partition $\mathcal{P}$ consisting of the sets $X_1,\ldots, X_n$, $Y_1, \ldots,Y_n$ is the partition of $V\Gamma$ into singleton sets. Note that $\mathcal{P}$ is invariant under $\Aut(\Gamma)$, and hence under $G$. This establishes \pref{Gamma-P}.  Note that  \pref{Gamma-flip} also clearly holds.

Since $L$ is transitive, $\Gamma$ is $G$-vertex-transitive. For $v\in X$, we have $G_v= L_v\times L$ with the elements in the first component acting trivially on $\Gamma(v)$. It follows that $(\Gamma,G)$ is locally-$L$, proving \pref{Gamma-LocalAction}. Note that the kernel of the action of $G$ on  $\{X,Y\}$  is  $S \coloneqq L\times L$. Parts \pref{Gamma-S-f} and \pref{Gamma-S-S^p} follow immediately. Finally, for $v\in X$ and $\pi\in L$, let $s\coloneqq (1,\pi)\in S$. This clearly satisfies \pref{Gamma-S-Stabiliser}.

This concludes the proof that the graph-group pair $(\Gamma,G)$ satisfies conditions \pref{Gamma-Bipartite}--\pref{Gamma-S}. In order to generate infinitely many pairs with the same properties but arbitrarily large graphs, one can take the homological $p$-cover of $\Gamma$ (see {\cite[Section 6]{ElemAbelCoversClassification}}) for infinitely many primes $p$, and the corresponding lifts of relevant objects (sets $X,Y,$ as well as $X_1,\ldots,X_n, Y_1,\ldots,Y_n$, groups $G$ and $S$ and the involution $f$) to these covers.
\end{proof}

\section{Proof of Theorem \pref{MainThm}}\label{ProofOfMainThm}
As in the statement of \cref{MainThm}, let $L$ be a transitive permutation group on a finite set $\Omega$ admitting a nontrivial block $B$ such that $L_{(\Omega \setminus B)}\neq 1$. We split the proof into several phases.
\[\textbf{Phase I: Embedding $L$ into $\SymGrp_k\wr\SymGrp_n$}\]

Let $\mathcal{B}$ be the block system for $L$ generated by $B$. Say that $\mathcal{B}$ consists of $n$ blocks of size $k$. Since $B$ is a proper block, we have $n\geq 2$. By the Embedding Theorem for Imprimitive Groups (see {\cite[Theorem 8.5]{InfPermGroups}}) there exists an embedding of permutation groups of the form
    \begin{equation}\label{EmbeddingOfL}
        L\hookrightarrow L_{\{B\}}^B\wr L^\mathcal{B}\leq \SymGrp_k \wr \SymGrp_n.
    \end{equation}
    
    This allows us to identify $\Omega$ with $\{1,\ldots,k\}\times \{1,\ldots,n\}$ in such a way that the blocks of $\mathcal{B}$ are given by $B_j \coloneqq \{(x,j) \colon 1\leq  x \leq k\}$ with $1\leq j\leq n$. Moreover, elements of $L$ can now be represented by tuples 
    \begin{equation}\label{ElementsOfL}
    \text{$((a_1,\ldots,a_n), a)$ with $a_i\in L_{\{B\}}^B\leq \SymGrp_k$ and $a\in L^\mathcal{B}\leq \SymGrp_n$.}
    \end{equation}

    We recall the rule for multiplication
    \begin{equation}\label{MultiplicationInL} 
        ((a_1,\ldots,a_n),a)((b_1,\ldots,b_n),b) = ((a_1b_{1^a},\ldots,a_nb_{n^a}),ab),
    \end{equation}

  as well as the rule for taking inverses
    \begin{equation}\label{InversionInL}
        ((a_1,\ldots,a_n),a)^{-1} = ((a_{1^{a^{-1}}})^{-1},\ldots,(a_{n^{a^{-1}}})^{-1},a^{-1}).
    \end{equation}

    Let $K$ be the kernel of the action of $L$ on the blocks of $\mathcal{B}$. The subgroup $K$ embeds via \pref{EmbeddingOfL} into the base group of the wreath product $\SymGrp_k \wr \SymGrp_n$, and therefore, an element $((a_1,\ldots,a_n),a)$ of $L$ is in $K$ if and only if $a=1$ (see \pref{ElementsOfL}). With this in mind, we will sometimes denote elements of $K$ by $(a_1,\ldots,a_n)$ with $a_i\in \SymGrp_k$ for  $1\leq i \leq n$. Furthermore, note that $L_{(\Omega\setminus B_1)}$ consists of the elements of $K$ of the form $(a,1,\ldots,1)$ for some $a\in \SymGrp_k$.

    As $K$ is a normal subgroup of $L$, conjugation by an arbitrary element $h =((c_1,\ldots,c_n),d)$ of $L$ induces an automorphism of $K$. This automorphism can be computed from \pref{MultiplicationInL} and \pref{InversionInL}. For an element $(a_1,\ldots,a_n)\in K$, we have that
    \begin{equation}\label{hiconjugation}
        (a_1,\ldots,a_n)^{h} = \left((a_{1^{d^{-1}}})^{c_{1^{d^{-1}}}},\ldots,(a_{n^{d^{-1}}})^{c_{n^{d^{-1}}}}\right).
    \end{equation}

    For each $i\in \{1,\ldots,n\}$, we fix an element $\ell_i\in L$ such that
    \begin{equation}\label{hidef}
        B_1^{\ell_i} = B_i.
    \end{equation}

   Without loss of generality, we can assume that
    \begin{equation}\label{hiaction}
        (x,1)^{\ell_i} = (x,i),\forall x \in \{1,\ldots,k\}, i\in \{1,\ldots,n\}.
    \end{equation}

    Using \pref{ElementsOfL}, we  can write
    \begin{equation}\label{hirelabelled}
        \ell_i = ((c_{1,i},\ldots,c_{n,i}),d_i) 
    \end{equation}
with $c_{j,i}\in \SymGrp_k$ for every $j\in \{1,\ldots,n\}$ and $d_i\in \SymGrp_n$.
    Using \pref{hiaction}, we find that $1^{d_i} = i$ and  $c_{1,i} = 1$. Consequently, if $a\in \SymGrp_k$ such that $(a,1,\dots,1)$ $\in$ $L_{(\Omega\setminus B_1)}\leq K$, then \pref{hiconjugation} implies that
    \begin{equation}\label{hishift}
        (a,1,\dots,1)^{\ell_i} = (1,\ldots,1,a ,1\ldots,1) \in K \text{ with $a$ in the $i^{th}$ position}.
    \end{equation}

    \[\textbf{Phase II: The graph $\Lambda$}\]
    
    Let $(\Gamma,G)$ be a graph-group pair obtained by applying \cref{Gamma} to the group $L^\mathcal{B}$. By \fullcref{Gamma}{Bipartite},  $\Gamma$ is connected, $n$-valent, and bipartite with bipartition sets $X$ and $Y$. By \fullcref{Gamma}{LocalAction}, the pair $(\Gamma,G)$ is locally-$L^\mathcal{B}$. By \fullcref{Gamma}{P}, $V\Gamma$ admits a $G$-invariant partition $\mathcal{P}$ into sets $X_i$ and $Y_i$ with $1\leq i\leq n$ such that
    \begin{equation*}
        X = \bigcup_{i=1}^{n} X_i \text{ and } Y = \bigcup_{i=1}^{n} Y_i.
    \end{equation*}
    
    Let $\Lambda \coloneqq \Gamma[\overline{K_k}]$ be the lexicographic product of $\Gamma$ with the edgeless graph on $k$ vertices. We label the vertices of $\Lambda$ by $V\Lambda \coloneqq \{1,\ldots,k\}\times V\Gamma$. The vertex set of $\Lambda$ admits a natural partition $\Sigma\coloneqq \{\Sigma_v \colon v \in V\Gamma\}$ into \textbf{bubbles} $\Sigma_v \coloneqq \{1,\ldots,k\}\times \{v\}$ for $v\in V\Gamma$.
    
\[\textbf{Phase III: Automorphisms of $\Lambda$}\]
    
   The group ${\SymGrp_k}\wr G$ acts naturally by automorphisms on $\Lambda$ admitting $\Sigma$ as a block system. Given $g\in G$, we define $\overline{g}$ to be the automorphism of $\Lambda$ such that for $(x,v)\in \{1,\ldots,k\}\times V\Gamma $, we have
    \begin{equation}\label{glift}
        (x,v)^{\overline{g}} \coloneqq (x,v^g).
    \end{equation}    

    Note that \pref{glift} defines a group monomorphism $G\rightarrow {\SymGrp_k}\wr G$ by the rule $g\mapsto \overline{g}$. For $a\in \SymGrp_k$ and $u\in V\Gamma$, we define an automorphism $[a]_u$ of $\Lambda$ by
    \begin{equation}\label{[a]v}
        (x,v)^{[a]_u}
        =
        \begin{cases*}
            (x^a,v), &  if $u = v$\\
            (x,v), & if $u \neq v$
        \end{cases*}
    \end{equation}
    Note that for $a,b\in \SymGrp_k$ and $u,v\in V\Gamma$ with $u\neq v$, we have that
    \begin{equation}\label{[a]vrelations}
        [a]_u[b]_u = [ab]_u \text{ and } [a]_u[b]_v = [b]_v [a]_u.
    \end{equation}

    Moreover, note that for $g\in G$, it holds that
    \begin{equation}\label{[a]vconjbyg}
        [a]_u^{\overline{g}} = [a]_{u^g}.
    \end{equation}

    For $a \in \SymGrp_k$ and $P\in\mathcal{P}$, we define the following automorphism of $\Lambda$
    \begin{equation}\label{chidef}
    \chi(a,P) \coloneqq \prod_{v\in P} [a]_v.
    \end{equation}
Note that, by \pref{[a]vrelations}, the factors of the product commute so this is well-defined.
    Then for every $a,b\in \SymGrp_k$, $P,P'\in\mathcal{P}$ with $P\neq P'$ and $g\in G$, the following hold:
        \begin{equation}\label{chiproduct}
            \chi(a,P)\chi(b,P) = \chi(ab,P),
        \end{equation}
        \begin{equation}\label{chicommuting}
            \text{$\chi(a,P)\chi(b,P')=\chi(b,P')\chi(a,P)$ },
        \end{equation}
        \begin{equation}\label{chiconjbyg}
            \chi(a,P)^{\overline{g}} = \chi(a,P^g).
        \end{equation}
        
\[\textbf{Phase IV: The group $N$}\]
Recall from \fullcref{Gamma}{S} that $S$ is the kernel of the action of $G$ on $\{X,Y\}$. We call a triple of elements $[\alpha,\beta,s]$ with $\alpha,\beta \in L$ and $s\in S$, \textbf{compatible} if, for every $i\in\{1,\ldots,n\}$, we have 
\begin{equation}\label{compatible}
X_i^s = X_{i^a} \text{ and } Y_i^s = Y_{i^b}, \text{ where }  a = \alpha^\mathcal{B} \text{ and } b = \beta^\mathcal{B}.
\end{equation} 
Let $\mathcal{C}$ denote the set of all compatible triples.

 \begin{claim}\label{CompatibleTriples}
For every $s\in S$, there exist $\alpha,\beta \in L$ such that $[\alpha,\beta,s]\in \mathcal{C}$.
 \end{claim}
 \begin{proof}[Proof of \cref{CompatibleTriples}]
     Let $s\in S$.  Then $s$ acts on the sets $X_1,\ldots,X_n$ and $Y_1,\ldots,Y_n$ inducing permutations $a,b\in\SymGrp_n$ such that $X_i^s = X_{i^a}$ and $Y_i^s = Y_{i^b}$ for every $i\in\{1,\ldots,n\}$. By \fullcref{Gamma}{S-S^p}, $a,b\in L^\mathcal{B}$, so in particular, there exist elements $\alpha,\beta \in L$ such that $\alpha^\mathcal{B} = a$ and $\beta^\mathcal{B}= b$, respectively. Hence, $[\alpha,\beta,s]\in \mathcal{C}$.
 \end{proof}
    We define $N$ to be the set of automorphisms of $\Lambda$ of the form 
\begin{equation}\label{ElementsOfN}
    N\coloneqq \left\{\left(\prod_{i=1}^n\chi(a_i,X_i)\prod_{i=1}^n\chi(b_i,Y_i)\right)\overline{s} \colon [\alpha,\beta,s]\in \mathcal{C}\right\}
\end{equation} where $\alpha = ((a_1,\ldots,a_n),a)$ and $\beta = ((b_1,\ldots,b_n),b)$  (see \pref{ElementsOfL}). It is clear from its definition that $N$ is a subset of ${\SymGrp_k} \wr G$. We show that it is, in fact, a subgroup.

\begin{claim}\label{NisaGroup}
    $N$ is a group.
\end{claim}
\begin{proof}[Proof of \cref{NisaGroup}]
    Note that $[1_L,1_L,1_S]$ is a compatible triple and hence, contained in $\mathcal{C}$. By \pref{ElementsOfN}, the corresponding element of $N$ is the identity permutation on $V\Lambda$. Hence, $N$ contains the identity.
    
    Let \[\left(\prod_{i=1}^n\chi(a_i,X_i)\prod_{i=1}^n\chi(b_i,Y_i)\right)\overline{s} 
    \text{ and } \left(\prod_{i=1}^n\chi(c_i,X_i)\prod_{i=1}^n\chi(d_i,Y_i)\right)\overline{t} \]

    be arbitrary elements of $N$ corresponding to compatible triples $[\alpha,\beta,s]$ and $[\gamma,\delta,t]$ in $\mathcal{C}$, respectively. In more detail, according to \pref{compatible}, this means that $s,t\in S$ and
    \begin{equation}\label{sandt}
        \text{$X_i^s = X_{i^a},Y_i^s = Y_{i^b},X_i^t = X_{i^c}, Y_i^t = Y_{i^d}$ for all $1\leq i\leq n$ }
    \end{equation}
    
     with $\alpha = ((a_1,\ldots,a_n),a), \beta = ((b_1,\ldots,b_n),b), \gamma = ((c_1,\ldots,c_n),c)$ and $\delta = ((d_1,\ldots,d_n),d)$ being elements of $L$.

    We consider the product of these two elements of $N$.
    \begin{equation}\label{ProductInN}
        \begin{split}
&\left(\left(\prod_{i=1}^n\chi(a_i,X_i)\prod_{i=1}^n\chi(b_i,Y_i)\right)\overline{s} \right)\left(\left(\prod_{i=1}^n\chi(c_i,X_i)\prod_{i=1}^n\chi(d_i,Y_i)\right)\overline{t} \right) = \\
&= \left(\prod_{i=1}^n\chi(a_i,X_i)\prod_{i=1}^n\chi(b_i,Y_i)\right)\left(\left(\prod_{i=1}^n\chi(c_i,X_i)\prod_{i=1}^n\chi(d_i,Y_i)\right)^{\overline{s}^{-1}}\right) \overline{st} =\\
 &= \left(\prod_{i=1}^n\chi(a_i,X_i)\prod_{i=1}^n\chi(b_i,Y_i)\right)\left(\prod_{i=1}^n\chi(c_i,X_i)^{{\overline{s}^{-1}}}\
    \prod_{i=1}^n\chi(d_i,Y_i)^{{\overline{s}^{-1}}}\right) \overline{st} \stackrel{\pref{chiconjbyg}}{=} \\
                        &= \left(\prod_{i=1}^n\chi(a_i,X_i)\prod_{i=1}^n\chi(b_i,Y_i)\right)\left(\prod_{i=1}^n\chi(c_i,X_i^{s^{-1}})
    \prod_{i=1}^n\chi(d_i,Y_i^{s^{-1}})\right) \overline{st} \stackrel{\pref{sandt}}{=} \\
     &= \left(\prod_{i=1}^n\chi(a_i,X_i)\prod_{i=1}^n\chi(b_i,Y_i)\right)\left(\prod_{i=1}^n\chi(c_i,X_{i^{a^{-1}}})
   \prod_{i=1}^n\chi(d_i,Y_{i^{b^{-1}}})\right) \overline{st} = \\
        &= \left(\prod_{i=1}^n\chi(a_i,X_i)\prod_{i=1}^n\chi(b_i,Y_i)\right)\left(\prod_{i=1}^n\chi(c_{i^a},X_i)
   \prod_{i=1}^n\chi(d_{i^b},Y_i)\right) \overline{st} \stackrel{\pref{chicommuting}}{=} \\
        &= \left(\prod_{i=1}^n\chi(a_i,X_i)\prod_{i=1}^n\chi(c_{i^a},X_i)
   \prod_{i=1}^n\chi(b_i,Y_i)\prod_{i=1}^n\chi(d_{i^b},Y_i)\right) \overline{st} \stackrel{\pref{chiproduct}}{=}\\
    &= \left(\prod_{i=1}^n\chi(a_ic_{i^a},X_i)
   \prod_{i=1}^n\chi(b_id_{i^b},Y_i)\right) \overline{st}.
        \end{split}
    \end{equation}

    Note that by \pref{MultiplicationInL}, we have that
    \[\alpha\gamma = ((a_1,\ldots,a_n),a)((c_1,\ldots,c_n),c) = ((a_1 c_{1^a},\ldots,a_nc_{n^a}),ac),\]
    and
    \[\beta\delta = ((b_1,\ldots,b_n),b)((d_1,\ldots,d_n),d) = ((b_1 d_{1^b},\ldots,b_n d_{n^b}),bd).\]

    Finally, by \pref{sandt}, we have that
    \[X_i^{(st)}= ((X_i)^s)^t = (X_{i^a})^t = X_{(i^a)^c} = X_{i^{ac}},\]
    and
    \[Y_i^{(st)}= ((Y_i)^s)^t = (Y_{i^b})^t = Y_{(i^c)^d} = Y_{i^{bd}}.\]

    Hence, $[\alpha\gamma,\beta\delta,st]\in\mathcal{C}$ is a compatible triple. Moreover, as per \pref{ElementsOfN}, the corresponding element of $N$ induced by $[\alpha\gamma,\beta\delta,st]$ is precisely the element obtained in \pref{ProductInN}. This proves that $N$ is closed under multiplication.  Since we are working inside a finite group, it follows that $N$ is a group.
\end{proof}

Let $w\in X_1\subseteq X\subseteq V\Gamma$. Note that $(1,w)\in V\Lambda$.

\begin{claim}\label{StabiliserInN}
     \[N_{(1,w)} = \left\{\left(\prod_{i=1}^n\chi(a_i,X_i)\prod_{i=1}^n\chi(b_i,Y_i)\right)\overline{s}\colon [\alpha,\beta,s]\in\mathcal{C}\mid 1^{a_1} = 1, w^s = w\right\}\]

     where $\alpha = ((a_1,\ldots,a_n),a)$ and $\beta = ((b_1,\ldots,b_n),b)$ in $L$.
\end{claim}
\begin{proof}[Proof of \cref{StabiliserInN}]
Since $w\in X_1$, we have

    \begin{equation}
        \begin{split}
    (1,w)^{\left(\prod_{i=1}^n\chi(a_i,X_i)\prod_{i=1}^n\chi(b_i,Y_i)\right)\overline{s}}
    &\stackrel{\pref{[a]v},\pref{chidef}}{=} (1^{a_1},w)^{\overline{s}} \stackrel{\pref{glift}}{=}(1^{a_1},w^s).
        \end{split}
    \end{equation}
     The conclusion follows.
\end{proof}

\begin{claim}\label{LocalActionOfN}
    $N_{(1,w)}^{\Lambda(1,w)}$ is permutation isomorphic to $L$.
\end{claim}
\begin{proof}[Proof of \cref{LocalActionOfN}]
    Label the neighbours of $w$ in $\Gamma$ by $u_1,\ldots,u_n$ such that
    \begin{equation}\label{ui}
        \{u_i\} = \Gamma(w)\cap Y_i,
    \end{equation}
    
    for all $1\leq i\leq n$ (such a labelling is possible by \fullcref{Gamma}{P}). Note that
    \begin{equation}\label{NLambda(1,w)}
   \Lambda(1,w) = \{1,\ldots,k\}\times \Gamma(w) = \bigcup_{i = 1}^n \Sigma_{u_i}.
    \end{equation}

    The identification of $(i,u_j)\in \Lambda(1,w)$ with  $(i,j)\in \{1,\ldots,k\}\times \{1,\ldots,n\}$ establishes that $N_{(1,w)}^{\Lambda(1,w)}$ is permutation isomorphic to a subgroup of $\SymGrp_k\wr\SymGrp_n$ acting on $\{1,\ldots,k\}\times \{1,\ldots,n\}$ with blocks $B_j= \{(i,j)\colon 1\leq i\leq k\}$ (where $B_j$ corresponds to the bubble $\Sigma_{u_j}$). We will show that this subgroup is precisely the embedded copy of $L$ we have defined in \pref{EmbeddingOfL}.

    By \cref{StabiliserInN}, elements of the stabiliser $N_{(1,w)}$ are of the form
    \[\left(\prod_{i=1}^n\chi(a_i,X_i)\prod_{i=1}^n\chi(b_i,Y_i)\right)\overline{s},\]

     where $s\in S$, $\alpha = ((a_1,\ldots,a_n),a)$ and $\beta = ((b_1,\ldots,b_n),b)$ are elements of $L$ with $X_i^s = X_{i^a},Y_i^s = Y_{i^b}$ (i.e., $[\alpha,\beta,s]$ is a compatible triple) and $1^{a_1} = 1$, $w^s = w$.
     
     Let $(p,u_q)$ be a  neighbour of $(1,w)$ in $\Lambda$ with $1\leq p\leq k$ and $1\leq q\leq n$ (see \pref{NLambda(1,w)}). Recall that $u_q\in Y_q$ (see \pref{ui}). We have that
     \begin{equation*}
         \begin{split}
(p,u_q)^{\left(\prod_{i=1}^n\chi(a_i,X_i)\prod_{i=1}^n\chi(b_i,Y_i)\right)\overline{s}} 
&\stackrel{\pref{[a]v},\pref{chidef}}{=} (p^{b_q},u_q)^{\overline{s}} \stackrel{\pref{glift}}{=} (p^{b_q},u_q^s).
         \end{split}
     \end{equation*}

     Since $s$ fixes $w$, $s$ also preserves $\Gamma(w)$. As $u_q$ is the unique neighbour of $w$ in $Y_q$ (see \pref{ui}) and $[\alpha,\beta,s]\in\mathcal{C}$ (see \pref{compatible}), we have that
    \begin{equation}\label{uq^s=uq^b}
        \begin{split}
            \{u_q^s\} = \{u_q\}^s=(\Gamma(w)\cap Y_q)^s = \Gamma(w)^s \cap Y_q^s = \Gamma(w) \cap Y_{q^b} =\{u_{q^b}\}
        \end{split}
    \end{equation}

    and thus $u_q^s = u_{q^b}$. We also have that
     \begin{equation}\label{puq}(p,u_q)^{\left(\left(\prod_{i=1}^n\chi(a_i,X_i)\prod_{i=1}^n\chi(b_i,Y_i)\right)\overline{s}\right)^{\Lambda(1,w)}} = (p^{b_q},u_{q^b}) 
     \end{equation}

    and
    \begin{equation}\label{pqbeta}
        (p,q)^\beta = (p,q)^{((b_1,\ldots,b_n),b)} = (p^{b_q},q^b).
    \end{equation}

    Comparing \pref{puq} and \pref{pqbeta} shows that by identifying $(p,u_q)$ with $(p,q)$, the element $\left(\left(\prod_{i=1}^n\chi(a_i,X_i)\prod_{i=1}^n\chi(b_i,Y_i)\right)\overline{s}\right)^{\Lambda(1,w)}$ of $N_{(1,w)}^{\Lambda(1,w)}$ is mapped onto $\beta = ((b_1,\ldots,b_n),b)$ in $L$. By a slight abuse of notation, this shows that $N_{(1,w)}^{\Lambda(1,w)}\leq L$.
    
    To show the converse, let $\beta = ((b_1,\ldots,b_n),b)\in L$. Then $b = \beta^\mathcal{B} \in L^\mathcal{B}$. By \fullcref{Gamma}{S-Stabiliser},  there exists  $s\in S$ such that $w^s = w$, $X_i^s = X_i$ and $Y_i^s = Y_{i^b}$ for all $1\leq i\leq n$. We conclude that $[1_L,\beta,s]\in \mathcal{C}$ is a compatible triple. Moreover, by \pref{ElementsOfN}, the corresponding element of $N$ is  \[\left(\prod_{i=1}^n\chi(1_L,X_i)\prod_{i=1}^n\chi(b_i,Y_i)\right)\overline{s} = \left(\prod_{i=1}^n\chi(b_i,Y_i)\right)\overline{s}.\] Note that by \cref{StabiliserInN}, it follows that $\left(\prod_{i=1}^n\chi(b_i,Y_i)\right)\overline{s}\in N_{(1,w)}$. 

Let $1\leq p\leq k$ and $1\leq q\leq n$. Since $w^s = w$, we can repeat the argument in \pref{uq^s=uq^b}, to obtain that $u_q^s = u_{q^b}$. It follows that 
    \[(p,u_q)^{\left(\left(\prod_{i=1}^n\chi(b_i,Y_i)\right)\overline{s}\right)^{\Lambda(1,w)}} = (p^{b_q},u_{q^b}).\]

    Since, $(p,q)^\beta = (p^{b_q},q^b)$, it follows that (the embedded copy of) $N_{(1,w)}^{\Lambda(1,w)}$ contains $\beta$, which proves that $N_{(1,w)}^{\Lambda(1,w)}\geq L$, concluding the proof.
\end{proof}

\[\textbf{Phase V: The group $D$}\]

Recall that by \cref{Gamma}, $\Gamma$ admits an automorphism $f$ of order $2$ such that $G = \langle S,f\rangle$ and $X_i^f = Y_i$, $Y_i^f = X_i$ for all $1\leq i \leq n$. We next define a new group
\begin{equation}\label{Ntildedef}
    D\coloneqq \langle N,\overline{f}\rangle.
\end{equation}

\begin{claim}\label{Ntildeprop}
   $(\Lambda,D)$ is a locally-$L$ pair.
\end{claim}
\begin{proof}[Proof of \cref{Ntildeprop}]
     We first show that $\overline{f}$ normalises $N$. By \pref{ElementsOfN}, a typical element of $N$ is induced by a compatible triple $[\alpha,\beta,s]\in\mathcal{C}$ (see \pref{compatible}) and is of the form $\left(\prod_{i=1}^n\chi(a_i,X_i)\prod_{i=1}^n\chi(b_i,Y_i)\right)\overline{s}$ with $s\in S$, $\alpha = ((a_1,\ldots,a_n),a)$, $\beta = ((b_1,\ldots,b_n),b)\in L$ satisfying $X_i^s = X_{i^a}, Y_i^s = Y_{i^b}$.

    Note that 
    \begin{equation}\label{Nconjbyf}
        \begin{split}
&\left(\left(\prod_{i=1}^n\chi(a_i,X_i)\prod_{i=1}^n\chi(b_i,Y_i)\right)\overline{s}\right)^{\overline{f}} = \left(\prod_{i=1}^n\chi(a_i,X_i)^{\overline{f}}\prod_{i=1}^n\chi(b_i,Y_i)^{\overline{f}}\right)\overline{s}^{\overline{f}} \\
&\stackrel{\pref{chiconjbyg}}{=} \left(\prod_{i=1}^n\chi(a_i,X_i^f)\prod_{i=1}^n\chi(b_i,Y_i^f)\right)\overline{s^f} =\left(\prod_{i=1}^n\chi(a_i,Y_i)\prod_{i=1}^n\chi(b_i,X_i)\right)\overline{s^f} \\
&\stackrel{\pref{chicommuting}}{=} \left(\prod_{i=1}^n\chi(b_i,X_i)\prod_{i=1}^n\chi(a_i,Y_i)\right)\overline{s^f} 
        \end{split}
    \end{equation}

    By \fullcref{Gamma}{S-f}, $s^f\in S$. Moreover, by properties of $f$ (summarised in \fullcref{Gamma}{flip}), we have that 
    \[X_i^{s^f} = ((X_i)^{f^{-1}})^{sf} = (Y_i^s)^f = (Y_{i^b})^f = X_{i^b},\]
    and
    \[Y_i^{s^f} = ((Y_i)^{f^{-1}})^{sf} = (X_i^s)^f = (X_{i^a})^f = Y_{i^a}.\]

    Hence, $[\beta,\alpha,s^f]$ is a compatible triple. Moreover, the corresponding element of $N$ it induces is precisely the element computed in \pref{Nconjbyf}, as per \pref{ElementsOfN}. This concludes the proof that $\overline{f}$ normalises $N$.

    Note that $\overline{f}$ is not in $N$ since, for example, $N$ preserves the sets $\{1,\ldots,k\}\times X$ and $\{1,\ldots,k\}\times Y$, while $\overline{f}$ swaps them. Since $f$ is an involution, so is $\overline{f}$ (see \pref{glift}). It follows from \pref{Ntildedef} that \[D = \langle N,\overline{f}\rangle = N\rtimes \langle \overline{f}\rangle, \]

    with the group $\langle \overline{f}\rangle$ acting semiregularly on $V\Lambda$. Note that this implies that the corresponding point-stabilisers inside $N$ and $D$ are equal i.e.,
    \begin{equation}\label{SameStabilisers}
        D_{(1,w)} = N_{(1,w)}.
    \end{equation}
    
    Hence, by \cref{LocalActionOfN}, it follows that 
    \begin{equation}\label{LocalActionOfNTilde}
        D_{(1,w)}^{\Lambda(1,w)} = N_{(1,w)}^{\Lambda(1,w)} \cong L.
    \end{equation}
     It remains to show that $D$ is transitive on $V\Lambda$. Note that, by \cref{LocalActionOfN},  $N_{(1,w)}^{\Lambda(1,w)}$ is permutation isomorphic to $L$ and, since $L$ is transitive, $N_{(1,w)}$ and thus also $D$ is transitive on  $\Sigma_{u_j}\subseteq \bigcup_{i=1}^n \Sigma_{u_i} = \Lambda(1,w)$ for all $1\leq j \leq n$ (see \pref{NLambda(1,w)}). It now suffices to show that $D$ is transitive on $\Sigma$, the set of all bubbles. Since $D$ is a subgroup ${\SymGrp_k}\wr G$ inside $\Aut(\Lambda)$, $\Sigma$ is in fact a block system for the action of $D$. An element of $N$ induced by the compatible triple $[\alpha,\beta,s]\in\mathcal{C}$ acts on  $\Sigma_v$ with $v\in V$ by
        \begin{equation}\label{Sigmav^N}
\Sigma_v^{\left(\prod_{i=1}^n\chi(a_i,X_i)\prod_{i=1}^n\chi(b_i,Y_i)\right)\overline{s}} =\Sigma_v^{\left(\prod_{i=1}^n\chi(b_i,Y_i)\right)\overline{s}} = \Sigma_v^{\overline{s}} = \Sigma_{v^s},
        \end{equation}
        
        while 
        \begin{equation}\label{Sigmav^f}
            \Sigma_v^{\overline{f}} = \Sigma_{v^f}.
        \end{equation}

    By \cref{CompatibleTriples}, every $s\in S$ occurs in a compatible triple in $\mathcal{C}$ (see \pref{compatible}), which then induces an element of $N$ (see \pref{ElementsOfN}). Since by \fullcref{Gamma}{S-f}, $G = \langle S,f\rangle$ and $G$ is transitive on $V\Gamma$, we conclude that $D$ is transitive on the bubbles of $\Sigma$, as desired. Therefore, $\Lambda$ is $D$-vertex-transitive and $(\Lambda,D)$ is a locally-$L$ pair.
\end{proof} 

\vspace{-5mm}

\[\textbf{Phase VI: The group $M$}\]

    For a vertex $v\in V\Gamma$ define
    \begin{equation}\label{Mvdef}
        M_v \coloneqq \{ [a]_v \colon (a,1,\ldots,1)\in L_{(\Omega\setminus B_1)}\}.
    \end{equation}

    Note that $M_v$ is a subgroup of $\Aut(\Lambda)$ isomorphic to $L_{(\Omega\setminus B_1)}$. Moreover, $M_v$ induces $L_{(\Omega\setminus B_1)}$ on the bubble $\Sigma_v$ and acts trivially on all other bubbles $\Sigma_u$ with $u\neq v$. From \pref{[a]v} and \pref{[a]vrelations}, we see that $M_u\cap M_v = 1$ and $[M_v,M_u] = 1$ for all $u,v\in V\Gamma$ with $u\neq v$. In particular, we have the following:
    \begin{equation}\label{Mdef}
        M\coloneqq\langle M_v \colon v\in V\Gamma\rangle = \prod_{v\in V\Gamma}M_v.
    \end{equation}
\[\textbf{Phase VII: The group $H$}\]
    We define a new group.
    \begin{equation}\label{Hdef}
        H\coloneqq \langle M,D\rangle.
    \end{equation}

    \begin{claim}\label{MnormalinH}
        $M$ is a normal subgroup of $H$.
    \end{claim}
    \begin{proof}[Proof of \cref{MnormalinH}]
        Since $H$ is generated by $M$ and $D$ (see \pref{Hdef}), and $D$ is generated by $N$ and $\overline{f}$ (see \pref{Ntildedef}), it suffices to show that $M$ is normalised by $N$ and $\overline{f}$.   A typical generator of $M$ is of the form $[c]_v$ with $(c,1,\ldots,1)\in L_{(\Omega\setminus B_1)}$ and $v\in V\Gamma$. By \pref{[a]vconjbyg}, it is clear that \[[c]_v^{\overline{f}} = [c]_{v^f} \in M_{v^f}\leq M.\] Hence, $\overline{f}$ normalises $M$.      Let 
        \[\left(\prod_{i=1}^n\chi(a_i,X_i)\prod_{i=1}^n\chi(b_i,Y_i)\right)\overline{s}\]

     be an arbitrary element of $N$ corresponding to a compatible triple $[\alpha,\beta,s]$ with $\alpha = ((a_1,\ldots,a_n),a)$ and $\beta = ((b_1,\ldots,b_n),b)$ elements of $L$.

     Suppose first that $v\in X_j$ for some $1\leq j\leq n$. Then we have the following:
     \begin{equation}\label{CvConjByN}
         \begin{split}
    [c]_v^{\left(\prod_{i=1}^n\chi(a_i,X_i)\prod_{i=1}^n\chi(b_i,Y_i)\right)\overline{s}}& \stackrel{\pref{[a]vrelations},\pref{chidef}}{=} ([c]_v^{[a_j]_v})^{\left(\prod_{i=1}^n\chi(b_i,Y_i)\right)\overline{s}}  \\
    &\stackrel{\pref{[a]vrelations},\pref{chidef}}{=}  {[c^{a_j}]_v}^{\left(\prod_{i=1}^n\chi(b_i,Y_i)\right)\overline{s}} \\
    &\stackrel{\pref{[a]vrelations}}{=}{[c^{a_j}]_v}^{\overline{s}} \\
    &\stackrel{\pref{[a]vconjbyg}}{=} [c^{a_j}]_{v^s}.
         \end{split}
     \end{equation}

     Note that to show the element obtained in \pref{CvConjByN} is in $M$, it suffices to show that $(c^{a_j},1\ldots,1)\in L_{(\Omega\setminus B_1)}$ (see \pref{Mvdef} and \pref{Mdef}). 
     \begin{equation}\label{ConjSteps}
        \begin{split}
             (c,1,\ldots,1)^{\ell_j\alpha(\ell_{j^a})^{-1}} &\stackrel{\pref{hishift}}{=} (1,\ldots,1,\underbrace{c}_{\text{ entry $j$}},1,\ldots,1)^{((a_1,\ldots,a_n),a)(\ell_{j^a})^{-1}} \\
             &\stackrel{\pref{hiconjugation}}{=}(1,\ldots,1,\underbrace{c^{a_j}}_{\text{ entry $j^a$}},1,\ldots,1)^{(\ell_{j^a})^{-1}}\\
             &\stackrel{\pref{hishift}}{=}(c^{a_j},1,\ldots,1).
        \end{split}
    \end{equation}

    As $(c,1,\ldots,1)\in L_{(\Omega\setminus B_1)}\leq K$ and $K$ is a normal subgroup of $L$, we conclude by \pref{ConjSteps} that $(c^{a_j},1\ldots,1)\in L_{(\Omega\setminus B_1)}$. Hence, the element calculated in \pref{CvConjByN} is indeed in $M$. 
    
    If instead $v\in Y_j$ for some $1\leq j\leq n$, analogous computations to \pref{CvConjByN} and \pref{ConjSteps} can be performed with $\beta = ((b_1,\ldots,b_n),b)$ instead of $\alpha = ((a_1,\ldots,a_n),a)$. Hence, $M$ is normal in $H$, as desired.
    \end{proof}

    By \cref{MnormalinH}, we can rewrite the group $H$ as
    \begin{equation}\label{H=MNtilde}
        H = \langle M,D\rangle = MD.
    \end{equation}
    
    \begin{claim}\label{StabiliserInH}
        $H_{(1,w)} = M_{(1,w)}D_{(1,w)} =M_{(1,w)}N_{(1,w)} $.
    \end{claim}
    \begin{proof}[Proof of \cref{StabiliserInH}]
        By \pref{SameStabilisers}, we know that $D_{(1,w)}  = N_{(1,w)}$, hence it follows that $M_{(1,w)}D_{(1,w)} =M_{(1,w)}N_{(1,w)}$. Moreover, by \pref{H=MNtilde} it is also clear that $M_{(1,w)}D_{(1,w)}\leq H_{(1,w)}$. Hence, to prove the claim, it suffices to show that $M_{(1,w)}D_{(1,w)}\geq H_{(1,w)}$. We will do this by showing that the groups $M_{(1,w)}D_{(1,w)}$ and $H_{(1,w)}$ are of the same order. Combining \pref{H=MNtilde} with the Orbit-Stabiliser lemma and the fact that $H$ is transitive on $V\Lambda$ (since $D$ is transitive on $V\Lambda$ by \cref{Ntildeprop}), we obtain the following
        \begin{equation}\label{OrderOfH1w}
            |H_{(1,w)}| = \frac{|H|}{|V\Lambda|} = \frac{|MD|}{|V\Lambda|} = \frac{|M||D|}{|M\cap D||V\Lambda|}.
        \end{equation}

        By several applications of the Orbit-Stabiliser lemma, and by the fact that $D$ is transitive on $V\Lambda$, we also have that
        \begin{equation}\label{OrderOfM1wN1w}
            \begin{split}
                |M_{(1,w)}D_{(1,w)}| &= \frac{|M_{(1,w)}||D_{(1,w)}|}{|M_{(1,w)}\cap D_{(1,w)}|} = \frac{|M_{(1,w)}||D_{(1,w)}|}{|(M\cap D)_{(1,w)}|}  = \\
                &=\frac{\frac{|M|}{|(1,w)^M|}\frac{|D|}{|(1,w)^{D}|}}{\frac{|M\cap D|}{|(1,w)^{M\cap D}|}} =\frac{|M||D|}{|M\cap D||V\Lambda|}\frac{|(1,w)^{M\cap D}|}{|(1,w)^M|}.
            \end{split}
        \end{equation}

        Combining \pref{OrderOfM1wN1w} with \pref{OrderOfH1w}, we have that
        \[ |M_{(1,w)}D_{(1,w)}| = |H_{(1,w)}|\frac{|(1,w)^{M\cap D}|}{|(1,w)^M|}.\]

        Hence, to finish the proof, it suffices to show that $|(1,w)^{M\cap D}| = |(1,w)^M|$. Since $M\cap D\leq M$, we in fact already know that \[(1,w)^{M\cap D} \subseteq (1,w)^M = (1,w)^{M_w} = \{(1^a,w) \colon (a,1,\ldots,1)\in L_{(\Omega\setminus B_1)}\}.\] We need to establish the reverse containment.

       Let $(a,1,\ldots,1)\in L_{(\Omega\setminus B_1)}$. Then by \pref{hishift}, we have that
       \[(a,\ldots,a) = \prod_{i=1}^n(a,1,\ldots,1)^{\ell_i} \in K \leq L.\]

       Note that $[(a,\ldots,a),1_L,1_S]\in \mathcal{C}$ is a compatible triple (see \pref{compatible}), which by \pref{ElementsOfN} induces the following element of $N$
       \begin{equation}\label{z}
           z \coloneqq \left(\prod_{i=1}^n\chi(a,X_i)\prod_{i=1}^n\chi(1_L,Y_i)\right)\overline{1_S} = \prod_{i=1}^n\chi(a,X_i) \stackrel{\pref{chidef}}{=} \prod_{v\in X}[a]_v\in N.
       \end{equation}

       From its definition, it is clear that $z \in M$ (see \pref{Mvdef} and \pref{Mdef}). Hence, $z \in M\cap N\leq M\cap D$ (since $N\leq D$ by \pref{Ntildedef}), and since $w\in X_1\subseteq X$, we have that
        \[(1^a,w) = (1,w)^z \in (1,w)^{M\cap D},\]
        as required.
        \end{proof}
 
Recall from \pref{NLambda(1,w)} that $\Lambda(1,w) = \bigcup_{i=1}^n\Sigma_{u_i}$ with $\Gamma(w) = \{u_1,\ldots,u_n\}$.

    \begin{claim}\label{LocMwinsideLocNw}
        $M_{(1,w)}^{\Lambda(1,w)}\leq N_{(1,w)}^{\Lambda(1,w)}$
    \end{claim}
    
    \begin{proof}[Proof of \cref{LocMwinsideLocNw}]
        By \pref{Mdef}, it suffices to show that $(M_{v})_{(1,w)}^{\Lambda(1,w)}\leq N_{(1,w)}^{\Lambda(1,w)}$ for all $v\in V\Gamma$.
        
        \begin{case}
            $v\notin \Gamma(w)$
        \end{case}
        
        From \pref{[a]v}, we see that $M_v$ acts trivially on each bubble $\Sigma_u$ with $u\neq v$. In particular, $M_v$ acts trivially on $\Lambda(1,w)$. Hence, in this case, it follows that \[(M_{v})_{(1,w)}^{\Lambda(1,w)} = 1 \leq N_{(1,w)}^{\Lambda(1,w)}.\]

        \begin{case}
            $v\in \Gamma(w)$ 
        \end{case}

        In this case, $v = u_j$ for some $j\in \{1,\ldots,n\}$. Note that $M_{u_j}$ fixes the vertex $(1,w)$, so $(M_{u_j})_{(1,w)} = M_{u_j}$. By \pref{Mvdef}, a typical generator of $M_{u_j}$ is given by $[a]_{u_j}$ for $(a,1,\ldots,1)\in L_{(\Omega\setminus B_1)}$. By \pref{hishift}, it follows that \[(a,1,\ldots,1)^{\ell_j} = (1,\ldots,\underset{\text{$j^{th}$ entry}}{a},1,\ldots,1)\in K\leq L.\]
        
        As $[1_L,(a,1,\ldots,1)^{\ell_j},1_S]\in \mathcal{C}$ is a compatible triple (see \pref{compatible}), by \pref{ElementsOfN} and \cref{StabiliserInN}, it induces an element of $N_{(1,w)}$ of the form
        \[ \left(\prod_{i=1}^n\chi(1_L,X_i)\right)\left(\chi(1,Y_1)\cdot \ldots \cdot \chi(a,Y_j)\cdot \ldots\cdot \chi(1,Y_n)\right)\overline{1_S} = \chi(a,Y_j).\]

        Moreover, $\chi(a,Y_j)$ acts trivially on $\{1,\ldots,k\}\times Y_i$ with $i\neq j$. As $u_j$ is the unique neighbour of $w$ in $Y_j$ (recall \pref{ui}), we conclude that $[a]_{u_j}$ and $\chi(a,Y_j)$ induce the same permutation on $\Lambda(1,w)$. Hence, \[(M_{u_j})_{(1,w)}^{\Lambda(1,w)}\leq N_{(1,w)}^{\Lambda(1,w)}.\]
    \end{proof}

    We summarise the most important properties of the group $H$ in the following claim.

    \begin{claim}\label{AllAboutH}
       The pair $(\Lambda,H)$ is locally-$L$. Moreover, it holds that
        \[|H_{(1,w)}| \geq |L_{(\Omega\setminus B_1)}|^{\frac{|V\Lambda|}{k}-1} \]
    \end{claim}
    \begin{proof}[Proof of \cref{AllAboutH}]
        From \pref{H=MNtilde}, we see that $H$ contains $D$ as a subgroup. By \cref{Ntildeprop}, we know that $D$ is vertex-transitive on $\Lambda$. Therefore, $H$ also is. We calculate the local action of $H$.
        \begin{equation*}
            \begin{split}
                H_{(1,w)}^{\Lambda(1,w)} & \stackrel{\cref{StabiliserInH}}{=} (M_{(1,w)}N_{(1,w)})^{\Lambda(1,w)} = M_{(1,w)}^{\Lambda(1,w)}N_{(1,w)}^{\Lambda(1,w)}  \\
                &\stackrel{\cref{LocMwinsideLocNw}}{=}N_{(1,w)}^{\Lambda(1,w)} \\
                &\stackrel{\cref{LocalActionOfN}}{\cong} L.
            \end{split}
        \end{equation*}

        Therefore, $(\Lambda,H)$ is a locally-$L$ pair.        Finally, we have the following lower bound on the order of a point-stabiliser in $H$.
        \begin{equation}\label{LowerBound}
            \begin{split}
                 |H_{(1,w)}| &\stackrel{\cref{StabiliserInH}}{=} |M_{(1,w)}D_{(1,w)}| \geq |M_{(1,w)}| \stackrel{\pref{Mdef}}{\geq} \\
                 &\geq \prod_{v\in V\Gamma \setminus\{w\}}|M_v| \stackrel{\pref{Mvdef}}{=} |L_{(\Omega\setminus B_1)}|^{|V\Gamma|-1} = \\
                 &=  |L_{(\Omega\setminus B_1)}|^{\frac{|V\Lambda|}{k}-1}.
            \end{split}
        \end{equation}
    \end{proof} 
    
\[\textbf{Phase VIII: Conclusion}\]
In Phase II, we started with a pair $(\Gamma,G)$   whose existence and properties were guaranteed by \cref{Gamma} and eventually defined the pair $(\Lambda,H)$ in terms of $(\Gamma,G)$. Note that \cref{Gamma} actually guarantees the existence of such pairs $(\Gamma,G)$ with $|V\Gamma|$ arbitrarily large. Constructing the pair $(\Lambda,H)$ corresponding to each such $(\Gamma,G)$ thus produces an infinite family of pairs $(\Lambda,H)$  for which \cref{AllAboutH} holds, with $|V\Lambda|$ arbitrarily large. In particular, these pairs are locally-$L$. Moreover we have $L_{(\Omega\setminus B)}\neq 1$ by assumption, which translates to $|L_{(\Omega\setminus B_1)}|\neq 1$ via \pref{EmbeddingOfL}. It follows from the bound in  \cref{AllAboutH} that the graph growth of $L$ is exponential, as desired.

\medskip

\begin{ack}
The first author is supported by the New Zealand Marsden Fund grant UOA-2122.
\end{ack}

\bibliography{NewRef}
\bibliographystyle{acm}

\end{document}